\documentclass[a4paper,11pt]{article}
\usepackage[latin1]{inputenc}
\usepackage[english]{babel}
\usepackage{amsmath}
\usepackage{amsfonts}
\usepackage{amssymb}
\usepackage{epsfig}
\usepackage{amsopn}
\usepackage{amsthm}
\usepackage{color}
\usepackage{graphicx}
\usepackage{subfigure}
\usepackage{enumerate}
\newtheorem{theorem}{Theorem}[section]

\newtheorem{proposition}[theorem]{Proposition}

\DeclareMathOperator\arctanh{arctanh}
\newtheorem*{theorem*}{Theorem}
\newtheorem*{lemma*}{Lemma}
\newtheorem*{remark*}{Remark}
\newtheorem*{definition*}{Definition}
\newtheorem*{proposition*}{Proposition}
\newtheorem*{corollary*}{Corollary}
\numberwithin{equation}{section}
\newcommand{\real}{\mathbb{R}}

\let\ced=\c         

\def\qed{\,\unskip\kern 6pt \penalty 500
\raise -2pt\hbox{\vrule \vbox to8pt{\hrule width 6pt
\vfill\hrule}\vrule}\par}
\definecolor{darkblue}{rgb}{0.05, .05, .65}
\definecolor{darkgreen}{rgb}{0.1, .65, .1}
\definecolor{darkred}{rgb}{0.8,0,0}
\newcommand{\beqn}{\begin{equation}}
\newcommand{\eeqn}{\end{equation}}
\newcommand{\bear}{\begin{eqnarray}}
\newcommand{\eear}{\end{eqnarray}}
\newcommand{\bean}{\begin{eqnarray*}}
\newcommand{\eean}{\end{eqnarray*}}
\begin{document}

\title{\huge \bf Qualitative properties of solutions to a generalized Fisher-KPP equation}

\author{
\Large Razvan Gabriel Iagar\,\footnote{Departamento de Matem\'{a}tica
Aplicada, Ciencia e Ingenieria de los Materiales y Tecnologia
Electr\'onica, Universidad Rey Juan Carlos, M\'{o}stoles,
28933, Madrid, Spain, \textit{e-mail:} razvan.iagar@urjc.es},\\
[4pt] \Large Ariel S\'{a}nchez,\footnote{Departamento de Matem\'{a}tica
Aplicada, Ciencia e Ingenieria de los Materiales y Tecnologia
Electr\'onica, Universidad Rey Juan Carlos, M\'{o}stoles,
28933, Madrid, Spain, \textit{e-mail:} ariel.sanchez@urjc.es}\\
[4pt] }
\date{}
\maketitle

\begin{abstract}
The following Fisher-KPP type equation
$$
u_t=Ku_{xx}-Bu^q+Au^p, \quad (x,t)\in\real\times(0,\infty),
$$
with $p>q>0$ and $A$, $B$, $K$ positive coefficients, is considered. For both $p>q>1$ and $p>1$, $q=1$, we construct stationary solutions, establish their behavior as $|x|\to\infty$ and prove that they are separatrices between solutions decreasing to zero in infinite time and solutions presenting blow-up in finite time. We also establish decay rates for the solutions that decay to zero as $t\to\infty$.
\end{abstract}

\

\noindent {\bf Mathematics Subject Classification 2020:} 35A17, 35B30, 35B44, 35K57, 35Q92.

\smallskip

\noindent {\bf Keywords and phrases:} Fisher-KPP equation, stationary solutions, separatrix, decay rates, finite time blow-up.

\section{Introduction}

The aim of this paper is to bring into light some qualitative properties related to the dynamics of solutions to the following generalized Fisher-KPP model:
\begin{equation}\label{eq0}
u_t=Ku_{xx}-Bu^q+Au^p,
\end{equation}
posed for $(x,t)\in\real\times(0,\infty)$, in the generic range of exponents $p>q>0$, $p>1$ and with $A$, $B$, $K$ positive coefficients. The main feature of Eq. \eqref{eq0} is the competition between three terms having different effects for large times: a source term with exponent $p>1$ which, when alone, usually leads to finite time blow-up (see for example \cite{QS} for a thorough study of this phenomenon), an absorption term with exponent $q>0$ which, when alone, implies a dissipation, that is, a loss of the $L^1$ norm (leading also to a finite time extinction if $q<1$), and a diffusion term which is conservative with respect to the $L^1$ norm. Thus, the balance between reaction and absorption will determine the main features related to the large time behavior of the solutions. In order to simplify the model, we can perform the following scaling
\begin{equation}\label{resc1}
x=a\overline{x}, \quad t=b\overline{t}, \quad u=c\overline{u},
\end{equation}
with coefficients
\begin{equation}\label{resc2}
c=\left(\frac{B}{A}\right)^{1/(p-q)}, \quad a=\left(\frac{Kc^{1-p}}{A}\right)^{1/2}, \quad b=\frac{c^{1-p}}{A},
\end{equation}
in order to obtain that, in the new variables $(\overline{x},\overline{t})$, the rescaled function $\overline{u}$ is a solution to
\begin{equation}\label{eq1}
\overline{u}_{\overline{t}}=(\overline{u})_{\overline{x}\overline{x}}-\overline{u}^q+\overline{u}^p.
\end{equation}
We shall thus work, without any loss of generality, with Eq. \eqref{eq1}, and for simplicity we drop from the notation the overlines, that is, we relabel the solution and the variables in Eq. \eqref{eq1} as $(u,x,t)$.  Eq. \eqref{eq1} is a very well-established model when the reaction and absorption exponents satisfy the opposite order, that is, $0<p<q$, stemming from the seminal works by Fisher \cite{Fi37} and Kolmogorov, Petrovsky and Piscounoff \cite{KPP37}, dealing with the specific case $p=1$ and $q=2$ as a model in mathematical biology. Nowadays, the Fisher-KPP equation with $0<p<q$ is rather well understood, new applications have been proposed and, from the mathematical point of view, it has been noticed that the dynamics of the solutions is well represented by solutions in form of traveling waves with a speed $c>0$, that is
\begin{equation}\label{TW}
u(x,t)=f(x-ct), \quad c>0,
\end{equation}
where $f(\cdot)$ is the profile of the wave, see for example \cite{AW75, GK, SHB05, Ma10, EGS13, DQZ20} and references therein. In particular, a very famous result is the existence of a critical speed $c^*>0$ such that traveling waves only exist if the speed is $c\geq c^*$ and do not exist when $c<c^*$.

\medskip

The opposite case of ordering between the reaction and absorption exponents, that is, $p>q>0$, has been also proposed, more recently, in a number of models arising in applied sciences. As a general model, Eq. \eqref{eq1} appears in growth and diffusion models as established, for example, in the book \cite{Ba94}. More specific applications in the mathematical modeling of tumor growth have been proposed by authors such as Marusic and Bajzer and their collaborators, see for example \cite{MB93, MBVF94, BMV96} and references therein, and a similar simplified model in \cite{VS84}.

Despite these applications, we discovered that there are still some gaps in the mathematical study of Eq. \eqref{eq1}. This is probably due to the fact that, in stark contrast to the more ``standard" Fisher-KPP equation, that is, with exponents ordered as $q>p$, in our case it appears that traveling waves (and other explicit or semi-explicit solutions, as we shall see in the present work) are unstable with respect to the dynamics of general solutions to Eq. \eqref{eq1}; that is, even if an initial condition $u_0\in C(\real)$ is sufficiently close (with respect to the $L^{\infty}$ norm) to one of several explicit solutions, it will evolve either by growing up (and then blowing up in finite time) or by decaying as $t\to\infty$. We mention here \cite[Chapter 17]{QS} where solutions to Eq. \eqref{eq1} with $q=1$ are studied and the short note \cite{HBS14} devoted to the range $p>1>q>0$ and $m>1$, where the transition from finite time extinction to blow-up is established. The existence of a separatrix in the form of a stationary solution has been thoroughly investigated for the supercritical semilinear heat equation in \cite{BF15, GNW92, GNW01, W93} (see also references therein), while some more general equations related to Eq. \eqref{eq1}, or particular cases of it, have been considered in \cite{Ha24, Zhang24}.

Putting into light this instability of several stationary solutions is actually the core of this work; indeed, we construct several solutions to Eq. \eqref{eq1} (in either explicit or implicit form) and we then prove that they serve as separatrices for the dynamics of general solutions to the Cauchy problem associated to Eq. \eqref{eq1} with suitable initial conditions.  More precisely, solutions with data lying \emph{above} the specific solution (even very close to it) \emph{blow up in finite time}, while solutions with data lying \emph{below} the specific solution \emph{decay as $t\to\infty$} (and even \emph{vanish in finite time} if $0<q<1$, as shown in \cite{HBS14}) and we give some decay estimates. To fix the notation, we consider throughout this work the following family of initial conditions:
\begin{equation}\label{ic}
u(x,0)=u_0(x), \quad x\in\real, \quad u_0\in L^{\infty}(\real)\cap C(\real).
\end{equation}
Since we are dealing with a semilinear problem, solutions will be taken in classical sense. We say that $u$ is a \emph{subsolution} (respectively \emph{supersolution}) to Eq. \eqref{eq1} if the sign of equality is replaced with $\leq$ (respectively $\geq$) in Eq. \eqref{eq1}. Thus, our main tool in the forthcoming proofs is the \emph{comparison principle}, which is a well established property of Eq. \eqref{eq1} and even of much more general models of analogous type, see for example \cite{DK12, DJ23}.

Another motivation for writing this paper is that, in a forthcoming work, we are able to map by some transformations a rather general family of a priori more complex partial differential equations into various cases of Eq. \eqref{eq1}, and we did not find a proper reference in literature giving the precise information we need on Eq. \eqref{eq1}. We thus decided to fill in this gap, at the same time the current work serving for completing the study of the ranges of $q$ not considered in the short note \cite{HBS14} devoted to absorption exponents $q\in(0,1)$.

Concerning the organization of the material, the paper is divided into two sections, related to, by order of appearance, the ranges $q=1$ and $q\in(1,p)$, followed by a short discussion at the end. The two main sections are further divided into subsections with the following scheme: a first subsection establishing some exact solutions (either explicit or implicit), and then one or two subsections establishing, on the one hand, the decay rate of solutions lying below the constant solution and, on the other hand, the character of separatrix of the stationary solution established in the first subsection in any of these cases. The \emph{main difficulty} stems from the fact that, for data $u_0$ which are very close to the stationary solution, in any of the cases, a rather fine construction of subsolutions and supersolutions is required. We are now in a position to begin our analysis.

\section{The exponent $q=1$}\label{sec.equal}

As indicated in the title, throughout this section, we fix $q=1$ in Eq. \eqref{eq1}, that is,
\begin{equation}\label{eq1.bis}
u_t=u_{xx}-u+u^p.
\end{equation}
We follow the plan mentioned at the end of the Introduction.

\subsection{Some explicit solutions}\label{subsec.expl.equal}

We derive below some exact solutions to Eq. \eqref{eq1.bis}, all of them in explicit form.

$\bullet$ \textbf{constant solution:} it is obvious that $u(x,t)=1$ is the unique non-trivial constant solution.

$\bullet$ \textbf{solutions depending only on time:} we look for solutions in the form $u(x,t)=h(t)$. With this ansatz, Eq. \eqref{eq1.bis} becomes the differential equation
$$
h'(t)=-h(t)+h(t)^p, \quad t>0,
$$
which can be integrated explicitly, leading to the family of explicit solutions
\begin{equation}\label{sol.time}
U(x,t;C)=h(t;C)=\left[1+Ce^{(p-1)t}\right]^{-1/(p-1)}, \quad C\in\real.
\end{equation}
Let us notice here that the behavior of $U(\cdot,\cdot;C)$ as $t\to\infty$ depends on the sign of $C$. Indeed, it is obvious that $U$ decreases as $t\to\infty$ if $C>0$, while it blows up in finite time if $C<0$, noticing that we recover the constant solution $U\equiv 1$ for $C=0$.

$\bullet$ \textbf{stationary solutions:} we look for solutions in the form $u(x,t)=g(x)$. With this ansatz, Eq. \eqref{eq1.bis} becomes the differential equation
\begin{equation}\label{interm1}
g''(x)-g(x)+g(x)^p=0, \quad x\in\real.
\end{equation}
By multiplying \eqref{interm1} by $g'(x)$ and integrating, we find that
\begin{equation}\label{interm2}
(g')^2(x)-g^2(x)+\frac{2}{p+1}g^{p+1}(x)=K, \quad K\in\real.
\end{equation}
Since we would like to work with solutions decaying to zero as $|x|\to\infty$, we let $K=0$ in \eqref{interm2}. After some easy manipulations and an implicit integration, we obtain
\begin{equation*}
C\pm x=-\frac{2}{p-1}\arctanh\sqrt{1-\frac{2g(x)^{p-1}}{p+1}},
\end{equation*}
which can be written in an explicit form as follows:
\begin{equation}\label{stat.sol.equal}
g(x;C)=\left\{\frac{p+1}{2}\left[1-\tanh^2\left(C+\frac{p-1}{2}x\right)\right]\right\}^{1/(p-1)}, \quad C\in\real,
\end{equation}
an expression similar to the ones identified at the end of \cite[Section 2]{HBS14}. Recalling the identity
$$
1-\tanh^2\theta=\frac{4}{2+e^{2\theta}+e^{-2\theta}}, \quad \theta\in\real,
$$
we deduce from \eqref{stat.sol.equal} that the stationary solutions $g(\cdot;C)$ have an exponential decay as $|x|\to\infty$, more precisely
$$
g(x;C)\sim K(C,p)e^{-|x|}, \quad K(C,p):=\left(\frac{p+1}{2}\right)^{1/(p-1)}e^{2C/(p-1)}, \quad C\in\real.
$$
In particular, the stationary solution belongs to $L^1(\real)$. In the forthcoming analysis, we will let for simplicity $C=0$ and we will work with the even stationary solution $g(\cdot;0)$, but, apart from longer and more tedious calculations, the separatrix property of every stationary solution $g(\cdot;C)$ will follow in a completely analogous way.

\subsection{Decay rate and large time behavior below the constant solution}\label{subsec.decay.equal}

The aim of this section is to show that solutions starting from data which are smaller than one decay as $t\to\infty$ and stabilize towards a profile linked with the heat equation.
\begin{theorem}\label{th.decay.equal}
Let $u_0\in C(\real)\cap L^{\infty}(\real)$ be such that $0<\|u_0\|_{\infty}<1$. Then there exists $C>0$ such that the solution $u$ to the Cauchy problem \eqref{eq1.bis}-\eqref{ic} satisfies
\begin{equation}\label{decay.equal}
\|u(t)\|_{\infty}\leq Ce^{-t}, \quad (x,t)\in\real\times(0,\infty).
\end{equation}
If furthermore $u_0\in L^1(\real)$, we have the following large time behavior for the solution $u$ to the Cauchy problem \eqref{eq1.bis}-\eqref{ic}
\begin{equation}\label{asympt.equal}
\lim\limits_{t\to\infty}t^{1/2}\|e^tu(t)-G(t)\|_{\infty}=0,
\end{equation}
where $G(t)$ is the heat kernel
$$
G(x,t)=\frac{M}{\sqrt{4\pi t}}e^{-|x|^2/4t}, \quad M=\|u_0\|_{1}.
$$
\end{theorem}
\begin{proof}
Assume first that $u_0\in L^{\infty}(\real)\cap C(\real)$ and $\|u_0\|_{\infty}\in(0,1)$. Pick $C_0>0$ such that
$$
\|u_0\|_{\infty}<(1+C_0)^{-1/(p-1)}.
$$
Recalling the solution $u(\cdot,\cdot;C_0)$ defined in \eqref{sol.time}, the comparison principle then entails that
\begin{equation}\label{interm22}
u(x,t)\leq U(x,t;C_0)=\left[1+C_0e^{(p-1)t}\right]^{-1/(p-1)}\leq C_0^{-1/(p-1)}e^{-t},
\end{equation}
for any $(x,t)\in\real\times(0,\infty)$, and thus we have proved the estimate \eqref{decay.equal}.

Set next
$$
w(x,t):=e^tu(x,t), \quad (x,t)\in\real\times(0,\infty).
$$
Notice that $w(x,0)=u(x,0)=u_0(x)$, for any $x\in\real$, and a straightforward calculation shows that $w$ is a solution to the following equation
\begin{equation}\label{interm4}
w_t=w_{xx}+e^{-(p-1)t}w^p, \quad (x,t)\in\real\times(0,\infty).
\end{equation}
Assume now that, moreover, $u_0\in L^1(\real)$. The equation \eqref{interm4} strongly suggests that an asymptotic simplification to the heat equation as $t\to\infty$ is expected to take place. In order to prove it in a rigorous way, the simplest path is to apply the stability theorem by Galaktionov and V\'azquez \cite{GV91, GV03}. Indeed, if we let $v$ to be the solution to the standard heat equation $v_t=v_{xx}$ with the same initial condition $v(x,0)=u_0(x)$, $x\in\real$, we easily observe that $w$ is a supersolution to this problem, so that the comparison principle applied to the heat equation, together with \eqref{interm22}, ensure that
\begin{equation}\label{interm5}
v(x,t)\leq w(x,t)\leq C_0^{-1/(p-1)}, \quad (x,t)\in\real\times(0,\infty).
\end{equation}
The estimates \eqref{interm5}, together with well known properties of the heat equation, readily imply that the hypothesis required for the application of the above mentioned stability theorem are in force. The stability theorem thus gives that the $\omega$-limit set of the orbits $w(\cdot;t)$ as $t\to\infty$ are contained in the solutions to the standard heat equation. However, since the initial condition is $u_0$, the uniqueness of the solution to the Cauchy problem for the heat equation implies that $w(t)$ approaches $v(t)$ as $t\to\infty$. Since $u_0\in L^{1}(\real)$, the well known asymptotic behavior as $t\to\infty$ for the integrable solutions to the heat equation leads to the convergence \eqref{asympt.equal}.
\end{proof}

\noindent \textbf{Remark.} We can observe by comparison with a suitable Gaussian function, which is a subsolution to \eqref{interm4}, that, even if $u_0\in L^{\infty}(\real)$ but $u_0\not\in L^1(\real)$, we have
$$
Ct^{-1/2}e^{-t}\leq \|u(t)\|_{\infty}, \quad t>0.
$$

\subsection{The stationary solution as a separatrix}\label{subsec.stat.equal}

As commented in Section \ref{subsec.expl.equal}, we fix for simplicity $C=0$ and denote by $g_0(x)=g(x;0)$, $x\in\real$, the stationary solution defined in \eqref{stat.sol.equal} with $C=0$. In this section, we prove that this solution (and, analogously or by a simple translation, any other stationary solution $g(\cdot;C)$) plays the role of a separatrix for the large time behavior of the solutions to the Cauchy problem \eqref{eq1.bis}-\eqref{ic}: that is, an initial condition strictly above it produces a solution whose $L^{\infty}$ norm increases with time (and in the end blows up in finite time), while an initial condition strictly below it produces a solution decaying in time, for which the outcome of Theorem \ref{th.decay.equal} applies. This is made precise in the next statement.
\begin{theorem}\label{th.separ.equal}
(a) Let $u_0\in L^{\infty}(\real)\cap C(\real)$ such that
\begin{equation}\label{inf.upper}
\inf\limits_{x\in\real}\frac{u_0(x)}{g_0(x)}=\kappa_0>1.
\end{equation}
Then the solution $u$ to the Cauchy problem \eqref{eq1.bis}-\eqref{ic} with initial condition $u_0$ blows up in finite time.

(b) Let $u_0\in L^{\infty}(\real)\cap C(\real)$ such that
\begin{equation}\label{sup.lower}
\sup\limits_{x\in\real}\frac{u_0(x)}{g_0(x)}=\kappa^0<1.
\end{equation}
Then the solution $u$ to the Cauchy problem \eqref{eq1.bis}-\eqref{ic} with initial condition $u_0$ decays to zero as $t\to\infty$ and behaves as in Theorem \ref{th.decay.equal}.
\end{theorem}

\noindent \textbf{Remark.} Before going to the proof, let us observe that conditions \eqref{inf.upper} and \eqref{sup.lower} are in fact related to a separation of the tails as $|x|\to\infty$. Indeed, due to the strong maximum principle, another solution cannot touch from above or from below the stationary solution $g_0$ at a time $t>0$, and for a fixed compact subset $K\subset\real$ one could really find $\kappa_0(K)$, respectively $\kappa^0(K)$ (depending in this case on the compact $K$) so that the previous conditions hold true in $K$. It is thus as $|x|\to\infty$ where a separation has to be required, as it does not follow from a maximum principle, and this is the sense of the two conditions \eqref{inf.upper} and \eqref{sup.lower}.

\begin{proof}[Proof of Theorem \ref{th.separ.equal}]
(a) Let us consider the function
$$
G(x,t)=(t+T)^{\delta}g_0((t+T)^{\gamma}x),
$$
where $T>1$, $\delta>0$ and $\gamma>0$ are parameters to be determined later, in order for $G$ to be a subsolution to the Cauchy problem \eqref{eq1.bis}-\eqref{ic}. A direct calculations gives
$$
G_t(x,t)=\delta(t+T)^{\delta-1}g_0(\zeta)+\gamma(t+T)^{\delta-1}\zeta g_0'(\zeta), \quad \zeta:=(t+T)^{\gamma}x.
$$
Since $g_0$ is an even function with a decreasing profile with respect to $x>0$, we deduce that
\begin{equation}\label{interm6}
G_t(x,t)\leq\delta(t+T)^{\delta-1}g_0(\zeta).
\end{equation}
Moreover, employing the equation \eqref{interm1} satisfied by $g_0$, we find that
\begin{equation}\label{interm7}
G_{xx}(x,t)=(t+T)^{\delta+2\gamma}g_0''(\zeta)=(t+T)^{\delta+2\gamma}(g_0(\zeta)-g_0^p(\zeta)).
\end{equation}
Gathering the outcome of \eqref{interm6} and \eqref{interm7}, we obtain
\begin{equation*}
\begin{split}
G_t(x,t)-G_{xx}(x,t)&+G(x,t)-G^p(x,t)\leq (t+T)^{\delta-1}\left[\delta+T+t-(t+T)^{2\gamma+1}\right]g_0(\zeta)\\
&+(t+T)^{\delta+2\gamma}\left[1-(t+T)^{\delta(p-1)-2\gamma}\right]g_0^p(\zeta).
\end{split}
\end{equation*}
We next choose the (up to now) free parameters as follows: fix first
\begin{equation}\label{deltagamma.equal}
0<\delta<\kappa_0^{(p-1)/2}-1, \quad \gamma=\frac{\delta(p-1)}{4}.
\end{equation}
With the previous choices, we are in a position to also choose
$$
T:=\kappa_0^{1/\delta}>1.
$$
Let us first observe that the terms in brackets in the previous calculation are negative with these choices, for any $t\geq0$. On the one hand, since $\delta(p-1)-2\gamma>0$, we immediately get that
$$
1-(t+T)^{\delta(p-1)-2\gamma}\leq 1-T^{\delta(p-1)-2\gamma}<0, \quad t\geq0.
$$
On the other hand, since $\gamma>0$, we have
\begin{equation*}
(T+t)^{2\gamma+1}-(t+T)=(t+T)((T+t)^{2\gamma}-1)\geq T(T^{2\gamma}-1)>\kappa_0^{(p-1)/2}-1>\delta,
\end{equation*}
whence
$$
\delta+T+t-(t+T)^{2\gamma+1}<0, \quad t\geq0.
$$
We have thus proved that $G(\cdot,t)$ is a subsolution to Eq. \eqref{eq1.bis} for any $t\geq0$. Moreover,
$$
G(x,0)=T^{\delta}g_0(T^{\gamma}x)=\kappa_0g_0(T^{\gamma}x)\leq u_0(T^{\gamma}x),
$$
hence $G$ is a subsolution to the Cauchy problem \eqref{eq1.bis}-\eqref{ic} with initial condition $u_0(T^{\gamma}x)$. The comparison principle then gives
$$
u((t+T)^{\gamma}x,t)\geq (t+T)^{\delta}g_0((t+T)^{\gamma}x), \quad (x,t)\in\real\times(0,\infty),
$$
or equivalently, in the new variable $\zeta$,
\begin{equation}\label{interm8}
u(\zeta,t)\geq (t+T)^{\delta}g_0(\zeta), \quad (\zeta,t)\in\real\times(0,\infty).
\end{equation}
Defining then the following energy
$$
E(u(t))=\frac{1}{2}\int_{\real}(|u_x|^2(x,t)+u^2(x,t))\,dx-\frac{1}{p+1}\int_{\real}u^{p+1}(x,t)\,dx,
$$
we observe that
\begin{equation*}
\begin{split}
E(G(t))&=\frac{1}{2}\int_{\real}\left[(T+t)^{2(\delta+\gamma)}|g_0'(x(T+t)^{\gamma})|^2\,dx+(T+t)^{2\delta}g_0(x(T+t)^{\gamma})^2\right]\,dx\\
&-\frac{(T+t)^{\delta(p+1)}}{p+1}\int_{\real}g_0^{p+1}(x(T+t)^{\gamma})\,dx\\
&=\frac{1}{2}(T+t)^{2\delta+\gamma}\int_{\real}|g_0'(y)|^2\,dy+\frac{1}{2}(T+t)^{2\delta-\gamma}\int_{\real}g_0^2(y)\,dy\\
&-\frac{1}{p+1}(T+t)^{(p+1)\delta-\gamma}\int_{\real}g_0^{p+1}\,dy<0,
\end{split}
\end{equation*}
provided $t>0$ is taken sufficiently large, since the choice of $\delta$ and $\gamma$ in \eqref{deltagamma.equal} implies
$$
(p+1)\delta-\gamma>2\delta+\gamma>2\delta-\gamma.
$$
We then infer from \cite[Theorem 17.6]{QS} that the solution to Eq. \eqref{eq1.bis} with initial condition $G(t_0)$ for $t_0$ sufficiently large such that $E(G(t_0))<0$ blows up in finite time. It follows by comparison that also $u$ blows up in finite time, as claimed.

\medskip

(b) Working ``in the mirror" with respect to part (a), let us consider the function
$$
H(x,t)=(t+T)^{-\delta}g_0((t+T)^{-\gamma}x),
$$
where $T>1$, $\delta>0$ and $\gamma>0$ are parameters to be determined later, in order for $H$ to be a supersolution to the Cauchy problem \eqref{eq1.bis}-\eqref{ic}. A direct calculations gives
$$
H_t(x,t)=-\delta(t+T)^{-\delta-1}g_0(\zeta)-\gamma(t+T)^{-\delta-1}\zeta g_0'(\zeta), \quad \zeta:=(t+T)^{-\gamma}x.
$$
Since $g_0$ is an even function with a decreasing profile with respect to $x>0$, we deduce that
\begin{equation}\label{interm6bis}
H_t(x,t)\geq-\delta(t+T)^{-\delta-1}g_0(\zeta).
\end{equation}
Moreover, employing the equation \eqref{interm1} satisfied by $g_0$, we find that
\begin{equation}\label{interm7bis}
H_{xx}(x,t)=(t+T)^{-\delta-2\gamma}g_0''(\zeta)=(t+T)^{-\delta-2\gamma}(g_0(\zeta)-g_0^p(\zeta)).
\end{equation}
Gathering the outcome of \eqref{interm6bis} and \eqref{interm7bis}, we obtain
\begin{equation*}
\begin{split}
H_t(x,t)-H_{xx}(x,t)&+H(x,t)-H^p(x,t)\geq (t+T)^{-\delta-1}\left[-\delta+T+t-(t+T)^{1-2\gamma}\right]g_0(\zeta)\\
&+(t+T)^{-\delta-2\gamma}\left[1-(t+T)^{-\delta(p-1)+2\gamma}\right]g_0^p(\zeta).
\end{split}
\end{equation*}
Proceeding as in part (a), we next choose the parameters as follows: fix first
\begin{equation}\label{deltagamma.equal.bis}
0<\delta<1-(\kappa^0)^{p-1}, \quad \gamma=\frac{\delta(p-1)}{2}.
\end{equation}
With the previous choices, we are in a position to also choose
$$
T:=(\kappa^0)^{-1/\delta}>1.
$$
Notice that $T^{-\delta}=\kappa^0$, whence
$$
H(x,0)=T^{-\delta}g_0(\zeta)=\kappa^0g_0(\zeta)\geq u_0(\zeta).
$$
Moreover, since $2\gamma=\delta(p-1)$ by \eqref{deltagamma.equal.bis}, we have
$$
1-(t+T)^{-\delta(p-1)+2\gamma}=0,
$$
for any $t\geq0$, while the first term in brackets is also non-negative, since
$$
T+t-(T+t)^{1-2\gamma}-\delta=(T+t)\left[1-(T+t)^{-2\gamma}\right]-\delta\geq 1-T^{-2\gamma}-\delta=1-(\kappa^0)^{(p-1)}-\delta>0,
$$
according to \eqref{deltagamma.equal.bis}. We infer that $H$ is a supersolution to the Cauchy problem \eqref{eq1.bis}-\eqref{ic} and thus, by comparison,
\begin{equation}\label{interm8bis}
u(\zeta,t)\leq (t+T)^{-\delta}g_0(\zeta), \quad (\zeta,t)\in\real\times(0,\infty).
\end{equation}
We then find from \eqref{interm8bis} that there is $t_0>0$ sufficiently large such that
$$
\|u(t_0)\|_{\infty}\leq (t_0+T)^{-\delta}\|g_0\|_{\infty}<1,
$$
and an application of Theorem \ref{th.decay.equal} starting with $t=t_0$ as initial time completes the proof.
\end{proof}

\section{The range $1<q<p$}\label{sec.upper}

Throughout this section, we work with absorption exponents $q\in(1,p)$. This choice does no longer allow for explicit integrations of the differential equations giving rise to solutions depending only on time or on the space variable, in contrast to the calculations in Section \ref{subsec.expl.equal}. Because of this technical problem, the forthcoming analysis is more involved than the previous one. We follow the same program as in Section \ref{sec.equal}.

\subsection{Some special solutions in implicit form}\label{subsec.expl.upper}

We examine below the properties of solutions either depending only on time or only on space, to Eq. \eqref{eq1}. We have:

$\bullet$ \textbf{constant solution:} once more, $u(x,t)\equiv1$ is a constant solution to Eq. \eqref{eq1}.

$\bullet$ \textbf{solutions depending only on time:} we look for solutions of the form $u(x,t)=h(t)$, $t>0$. With this ansatz, Eq. \eqref{eq1} becomes
\begin{equation}\label{interm9}
h'(t)=h^p(t)-h^q(t).
\end{equation}
It is easy to observe that $h(0)>1$ implies $h'(t)>0$ for any $t>0$, while $0<h(0)<1$ implies $h'(t)<0$ for any $t>0$. We next give a more precise description of the properties of solutions to \eqref{interm9}. Assume first that $h(0)>1$, hence $h(t)>h(0)$ for any $t>0$ and thus
$$
h^p(t)>h'(t)=h^p(t)(1-h^{q-p}(t))>h^p(t)(1-h(0)^{q-p}).
$$
The second inequality already implies finite time blow-up of $h$, and let us denote by $T\in(0,\infty)$ its blow-up time. We have
$$
1-h(0)^{q-p}<h^{-p}(t)h'(t)<1
$$
and a straightforward argument of integration on $(t,T)$ leads to the blow-up rate
\begin{equation}\label{blowup.rate.upper}
(p-1)^{-1/(p-1)}(T-t)^{-1/(p-1)}<h(t)<[(p-1)(1-h(0)^{q-p})]^{-1/(p-1)}(T-t)^{-1/(p-1)}.
\end{equation}
Assume now that $h(0)<1$. Then we get from \eqref{interm9} that
$$
-h^q(t)<h'(t)<(h(0)^{p-q}-1)h^q(t)<0,
$$
or equivalently
$$
(1-q)(h(0)^{p-q}-1)<(h^{1-q})'(t)<-(1-q)
$$
and by integration on $(0,t)$ and straightforward manipulations, we deduce the decay rate of the solution $h$ as $t\to\infty$ as follows:
\begin{equation}\label{decay.rate.upper}
[(q-1)t+h(0)^{1-q}]^{-1/(q-1)}<h(t)<[(q-1)(1-h(0)^{p-q})t+h(0)^{1-q}]^{-1/(q-1)},
\end{equation}
which means in particular that $h(t)$ decays like $t^{-1/(q-1)}$ as $t\to\infty$.

$\bullet$ \textbf{stationary solutions:} we look for solutions of the form $u(x,t)=\psi(x)$, $x\in\real$. With this ansatz, Eq. \eqref{eq1} becomes
$$
\psi''(x)-\psi^q(x)+\psi^p(x)=0,
$$
which, after a multiplication by $\psi'(x)$ and an obvious integration term by term (setting the integration constant to zero), reads
\begin{equation}\label{interm10}
(\psi')^2(x)=\frac{2}{q+1}\psi^{q+1}(x)-\frac{2}{p+1}\psi^{p+1}(x).
\end{equation}
Notice first that, at a maximum point (which we assume to be at $x=0$, as we want an even solution for simplicity), we obtain from \eqref{interm10}
\begin{equation}\label{interm11}
\psi(0)=\|\psi\|_{\infty}=\left(\frac{p+1}{q+1}\right)^{1/(p-q)}>1.
\end{equation}
The standard existence and uniqueness theorem applied to the Cauchy problem \eqref{interm10} (taking the square root with a minus sign in front) with initial condition \eqref{interm11} gives the existence of a stationary solution for $x>0$, which can be then extended by symmetry to $x<0$, having its maximum at $x=0$ and a decreasing profile, as desired. From now on, we denote this stationary solution by $\psi_0$. Its asymptotic properties are stated in the following result.
\begin{proposition}\label{prop.stat}
The following behavior of $\psi_0$ as $|x|\to\infty$ holds true:
\begin{equation}\label{limits}
\lim\limits_{|x|\to\infty}|x|^{2/(q-1)}\psi_0(x)=\left[\frac{2}{q-1}\sqrt{\frac{q+1}{2}}\right]^{2/(q-1)}, \quad \lim\limits_{|x|\to\infty}\frac{|x|\psi_0'(x)}{\psi_0(x)}=-\frac{2}{q-1}.
\end{equation}
\end{proposition}
\begin{proof}
Taking into account the radial symmetry, we can work on the half-plane $x>0$. We deduce on the one hand that
\begin{equation}\label{interm12}
\psi_0'(x)=-\sqrt{\frac{2}{q+1}\psi_0^{q+1}(x)-\frac{2}{p+1}\psi_0^{p+1}(x)}>-\sqrt{\frac{2}{q+1}}\psi_0^{(q+1)/2}(x),
\end{equation}
and on the other hand, for $x>1$,
\begin{equation}\label{interm12bis}
\begin{split}
\psi_0'(x)&=-\sqrt{\frac{2}{q+1}\psi_0^{q+1}(x)\left[1-\frac{q+1}{p+1}\psi_0^{p-q}(x)\right]}\\
&<-\sqrt{\frac{2}{q+1}\left[1-\frac{q+1}{p+1}\psi_0(1)^{p-q}\right]}\psi_0^{(q+1)/2}(x).
\end{split}
\end{equation}
Gathering the estimates \eqref{interm12} and \eqref{interm12bis}, we deduce that there exist $C_1>C_2>0$ such that
$$
-C_1\psi_0^{(q+1)/2}(x)<\psi_0'(x)<-C_2\psi_0^{(q+1)/2}(x), \quad x\in(1,\infty).
$$
Multiplying the previous inequalities by $\psi_0^{-(q+1)/2}(x)$ and integrating on $(1,x)$ leads to the bounds
\begin{equation*}
[\overline{C}_1(x-1)+\psi_0(1)^{(1-q)/2}]^{-2/(q-1)}<\psi_0(x)<[\overline{C}_2(x-1)+\psi_0(1)^{(1-q)/2}]^{-2/(q-1)}, \quad x>1,
\end{equation*}
with $\overline{C}_i=(q-1)C_i/2$, $i=1,2$, which in particular give that
\begin{equation}\label{bounds}
0<\overline{C}_1^{-2/(q-1)}<\liminf\limits_{x\to\infty}x^{2/(q-1)}\psi_0(x)\leq\limsup\limits_{x\to\infty}x^{2/(q-1)}\psi_0(x)<\overline{C}_2^{-2/(q-1)}.
\end{equation}
Set now $\phi_0(x):=x^{2/(q-1)}\psi_0(x)$. Straightforward calculations lead to the differential equation solved by $\phi_0$, that is,
\begin{equation}\label{interm21}
x\phi_0'(x)-\frac{2}{q-1}\phi_0(x)=-\phi_0^{(q+1)/2}(x)\sqrt{\frac{2}{q+1}-\frac{2}{p+1}x^{-2(p-q)/(q-1)}\phi_0^{p-q}(x)}.
\end{equation}
We next have two possibilities:

$\bullet$ either there exists $\lim\limits_{x\to\infty}\phi_0(x)=L\in(0,\infty)$, and in this case, an application of a standard calculus fact (see \cite[Lemma 2.9]{IL13} for a precise statement and proof) gives that there is a sequence $(x_j)_{j\geq1}$ such that $x_j\to\infty$ and $x_j\phi_0'(x_j)\to0$ as $j\to\infty$. Since $\phi_0(x_j)\to L$ as $j\to\infty$, we obtain by evaluating \eqref{interm21} at $x=x_j$ and letting $j\to\infty$ that
$$
-\frac{2}{q-1}L=-L^{(q+1)/2}\sqrt{\frac{2}{q+1}}, \quad {\rm that \ is}, \quad L=\left[\frac{2}{q-1}\sqrt{\frac{q+1}{2}}\right]^{2/(q-1)},
$$
as claimed.

$\bullet$ or the limit of $\phi_0(x)$ as $x\to\infty$ does not exist. Since $\phi_0$ is bounded according to \eqref{bounds}, it follows that it oscillates infinitely many times between two extremal values and thus there are sequences $(x^m_j)_{j\geq1}$ and $(x^M_j)_{j\geq1}$ of local minima, respectively local maxima for $\phi_0$, such that $x^m_j\to\infty$, $x^M_{j}\to\infty$ as $j\to\infty$ and that
$$
\lim\limits_{j\to\infty}\phi_0(x^m_j)=\liminf\limits_{x\to\infty}\phi_0(x), \quad \lim\limits_{j\to\infty}\phi_0(x^M_j)=\limsup\limits_{x\to\infty}\phi_0(x).
$$
By evaluating \eqref{interm21} at $x=x^m_j$, respectively $x=x^M_j$, and letting $j\to\infty$, we readily obtain that
$$
\liminf\limits_{x\to\infty}\phi_0(x)=\limsup\limits_{x\to\infty}\phi_0(x)=\left[\frac{2}{q-1}\sqrt{\frac{q+1}{2}}\right]^{2/(q-1)},
$$
which is a contradiction with the non-existence of the limit, showing that this case is not possible. Thus, the proof of the first limit in \eqref{limits} is complete. For the second limit in \eqref{limits}, it is enough to replace $\psi_0'(x)$ by the right hand side in \eqref{interm12}, obtaining thus an expression depending only on $\psi_0(x)$, and then employ the first limit in \eqref{limits} to get the result. We leave the easy details to the reader.
\end{proof}

\subsection{The stationary solution as a separatrix}\label{subsec.stat.upper}

Similarly as we did in Section \ref{subsec.stat.equal}, but technically more involved as we work with non-explicit solutions, we show next that the stationary solution constructed in the previous section plays the role of a separatrix for the qualitative properties of solutions to the Cauchy problem \eqref{eq1}-\eqref{ic}. This is made precise in the following statement.
\begin{theorem}\label{th.separ.upper}
(a) Let $u_0\in L^{\infty}(\real)\cap C(\real)$ be an initial condition such that $\|u_0\|_{\infty}<1$. Then the solution $u$ to the Cauchy problem \eqref{eq1}-\eqref{ic} decays as $t\to\infty$ and, more precisely, there exists $C>0$ such that
\begin{equation}\label{decay.upper}
\|u(t)\|_{\infty}\leq Ct^{-1/(q-1)}, \quad t>0.
\end{equation}

(b) Let $u_0\in L^{\infty}(\real)\cap C(\real)$ such that
\begin{equation}\label{inf.upper.upper}
\inf\limits_{x\in\real}\frac{u_0(x)}{\psi_0(x)}=\kappa_0>1.
\end{equation}
Then the solution $u$ to the Cauchy problem \eqref{eq1}-\eqref{ic} with initial condition $u_0$ blows up in finite time.

(c) Let $u_0\in L^{\infty}(\real)\cap C(\real)$ such that
\begin{equation}\label{sup.lower.upper}
\sup\limits_{x\in\real}\frac{u_0(x)}{\psi_0(x)}=\kappa^0<1.
\end{equation}
Then the solution $u$ to the Cauchy problem \eqref{eq1}-\eqref{ic} with initial condition $u_0$ decays to zero as $t\to\infty$ as in \eqref{decay.upper}.
\end{theorem}
\begin{proof}
(a) This follows directly by comparison with a solution $h(t)$ to \eqref{interm9} with $\|u_0\|_{\infty}<h(0)<1$. The comparison principle and the estimate \eqref{decay.rate.upper} lead to the conclusion.

\medskip

(b) Following similar ideas as in the proof of Theorem \ref{th.separ.equal}, we construct a subsolution to the Cauchy problem \eqref{eq1}-\eqref{ic} in the form
$$
\Psi(x,t)=(T+t)^{\delta}\psi_0((T+t)^{\gamma}x), \quad (x,t)\in[0,\infty)\times[0,\infty),
$$
with $T>1$, $\delta>0$ and $\gamma>0$ to be determined. Setting $\zeta=(T+t)^{\gamma}x$, we obtain by direct calculations
\begin{equation}\label{interm13}
\begin{split}
\Psi_t(x,t)&-\Psi_{xx}(x,t)+\Psi^q(x,t)-\Psi^p(x,t)=\delta(T+t)^{\delta-1}\psi_0(\zeta)\\
&+\gamma(T+t)^{\delta-1}\zeta\psi_0'(\zeta)-(t+T)^{\delta+2\gamma}\psi_0^q(\zeta)+(t+T)^{\delta+2\gamma}\psi_0^{p}(\zeta)\\
&+(t+T)^{q\delta}\Psi_0^q(\zeta)-(T+t)^{p\delta}\psi_0^p(\zeta)\\
&=(T+t)^{\delta-1}\left[\delta-\gamma\zeta\sqrt{\frac{2}{q+1}\psi_0^{q-1}(\zeta)-\frac{2}{p+1}\psi_0^{p-1}(\zeta)}\right]\psi_0(\zeta)\\
&+(T+t)^{\delta+2\gamma}\left[1-(t+T)^{\delta(p-1)-2\gamma}\right]\psi_0^{p}(\zeta)\\
&+(T+t)^{q\delta}\left[1-(t+T)^{2\gamma-\delta(q-1)}\right]\psi_0^q(\zeta)=T_1+T_2+T_3,
\end{split}
\end{equation}
where in the formula for the term $T_1$ we have employed the equation \eqref{interm10} satisfied by $\psi_0'(\zeta)$ for $\zeta\in(0,\infty)$. Since we want that our function $\Psi$ remains a subsolution for any $t>0$, the expressions of the terms $T_2$ and $T_3$ in \eqref{interm13} imply that a necessary condition is to fix
\begin{equation}\label{interm14}
\frac{\delta(q-1)}{2}<\gamma<\frac{\delta(p-1)}{2}.
\end{equation}
We next fix $T$ as follows:
\begin{equation}\label{interm15}
T=\kappa_0^{1/\delta}>1,
\end{equation}
remaining thus to choose $\delta$ as the only (still) free parameter. We infer from \eqref{interm14} and \eqref{interm15} that, on the one hand,
$$
(t+T)^{2\gamma-\delta(q-1)}\geq T^{2\gamma-\delta(q-1)}>1, \quad t\geq0
$$
and, on the other hand,
$$
(t+T)^{\delta(p-1)-2\gamma}\geq T^{\delta(p-1)-2\gamma}>1, \quad t\geq0,
$$
whence $T_2<0$ and $T_3<0$ in \eqref{interm13}, for any $t\geq0$. We are left with estimating the term $T_1$ in \eqref{interm13}. We derive from \eqref{limits} that
$$
\zeta\sqrt{\frac{2}{q+1}\psi_0^{q-1}(\zeta)-\frac{2}{p+1}\psi_0^{p-1}(\zeta)}=-\frac{\zeta\psi_0'(\zeta)}{\psi_0(\zeta)}\to\frac{2}{q-1},
$$
as $\zeta\to\infty$. Since $\delta/\gamma<2/(q-1)$ by \eqref{interm14}, we deduce that there exists $R_0>0$ such that
\begin{equation}\label{interm16}
\frac{\delta}{\gamma}<\zeta\sqrt{\frac{2}{q+1}\psi_0^{q-1}(\zeta)-\frac{2}{p+1}\psi_0^{p-1}(\zeta)}, \quad |\zeta|>R_0,
\end{equation}
which in particular implies that $T_1<0$ if $|\zeta|>R_0$. Set next
\begin{equation}\label{interm20}
L_0:=\inf\limits_{\zeta\in[-R_0,R_0]}\psi_0(\zeta)>0,
\end{equation}
and notice that $R_0$ and $L_0$ depend only on $p$ and $q$, if we fix, for example,
\begin{equation}\label{interm14bis}
\frac{\delta}{\gamma}=\frac{1}{p-1}+\frac{1}{q-1},
\end{equation}
a choice that satisfies \eqref{interm14}. We are left with estimating $T_1$ in \eqref{interm13} in the closed interval $[-R_0,R_0]$, and to this end we change a bit the strategy: instead of estimating $T_1$ alone, we recall that $\zeta\psi_0'(\zeta)\leq0$ for any $\zeta\in\real$ and we compensate $T_1$ with the negativity of $T_2$ in this interval. More precisely, discarding the already negative contribution of $\gamma(T+t)^{\delta-1}\zeta\psi_0'(\zeta)$, we want to choose $\delta>0$ such that, for any $\zeta\in[-R_0,R_0]$ and $t\geq0$, we have
$$
\delta(T+t)^{\delta-1}-(T+t)^{\delta+2\gamma}\left[(T+t)^{\delta(p-1)-2\gamma}-1\right]\psi_0^{p-1}(\zeta)<0,
$$
for which a sufficient condition is to pick $\delta$ such that
\begin{equation}\label{interm17}
0<\delta<\left(\kappa_0^{p-1-2\gamma/\delta}-1\right)L_0^{p-1}=\left(\kappa_0^{(p-1)(p-q)/(p+q-2)}-1\right)L_0^{p-1},
\end{equation}
taking into account \eqref{interm14bis} and that $T>1$ and $\kappa_0>1$. It thus follows that $\Psi$ is a subsolution to Eq. \eqref{eq1} and also satisfies
$$
\Psi(x,0)=\kappa_0\psi_0(T^{\gamma}x)\leq u_0(T^{\gamma}x),
$$
hence the comparison principle entails that
$$
\Psi(x,t)=(T+t)^{\delta}\psi_0(\zeta)\leq u(\zeta,t), \quad (\zeta,t)\in\real\times(0,\infty).
$$
In order to prove the finite time blow-up, we adapt the argument at the end of the proof of Theorem \ref{th.separ.equal} with a slightly changed energy functional in order to cope with the term $u^q$ instead of $u$, that is,
$$
E(u(t))=\frac{1}{2}\int_{\real}|u_x|^2(x,t)\,dx+\frac{1}{q+1}\int_{\real}u^{q+1}(x,t)\,dx-\frac{1}{p+1}\int_{\real}u^{p+1}(x,t)\,dx.
$$
The proof of \cite[Theorem 17.6, (ii)]{QS} straightforwardly adapts to this energy to show that, if for a general solution $u$ to Eq. \eqref{eq1} in $\real\times(0,\infty)$ we have $E(u(t))<0$ for some $t\geq0$, then the solution $u$ blows up in finite time. In our case, similarly as at the end of the proof of Theorem \ref{th.separ.equal}, we compute
\begin{equation*}
\begin{split}
E(\Psi(t))&=\frac{1}{2}(T+t)^{2\delta+\gamma}\int_{\real}|\psi_0'(y)|^2\,dy+\frac{1}{q+1}(T+t)^{\delta(q+1)-\gamma}\int_{\real}\psi_0^{q+1}(y)\,dy\\
&-\frac{1}{p+1}(T+t)^{\delta(p+1)-\gamma}\int_{\real}\psi_0^{p+1}(y)\,dy<0,
\end{split}
\end{equation*}
provided $t>t_0>0$ is sufficiently large, since the fact that $p>q$ and \eqref{interm14} ensure that
$$
(p+1)\delta-\gamma>(q+1)\delta-\gamma, \quad (p+1)\delta-\gamma>2\delta+\gamma.
$$
Thus, the solution to Eq. \eqref{eq1} with initial condition $\Psi(t_0)$ such that $E(\Psi(t_0))<0$ blows up in finite time. We conclude then from the comparison principle that $u$ also blows up in finite time.

\medskip

(c) We want to construct a supersolution to the Cauchy problem \eqref{eq1}-\eqref{ic} decaying in time, and we work again ``in the mirror" with respect to the construction performed in part (b), so that we will only give a sketch of it below. We consider
$$
\Phi(x,t)=(T+t)^{-\delta}\psi_0(x(T+t)^{-\gamma}), \quad \zeta=(T+t)^{-\gamma}x,
$$
and we want to choose the parameters and exponents $T$, $\delta$, $\gamma$ such that $\Phi$ is a supersolution. Direct calculations similar to the ones performed in part (b) lead to the following expression
\begin{equation}\label{interm13bis}
\begin{split}
\Phi_t(x,t)&-\Phi_{xx}(x,t)+\Phi^q(x,t)-\Phi^p(x,t)=-\gamma(T+t)^{-\delta-1}\left[\frac{\delta}{\gamma}+\frac{\zeta\psi_0'(\zeta)}{\psi_0(\zeta)}\right]\psi_0(\zeta)\\
&+(T+t)^{-\delta-2\gamma}\left[1-(t+T)^{-\delta(p-1)+2\gamma}\right]\psi_0^{p}(\zeta)\\
&+(T+t)^{-q\delta}\left[1-(t+T)^{-2\gamma+\delta(q-1)}\right]\psi_0^q(\zeta)=T_1+T_2+T_3.
\end{split}
\end{equation}
Since we want a supersolution for any $t\geq0$, we are forced to choose $\delta$, $\gamma$ such that the powers of $T+t$ in the second and third line of \eqref{interm13bis} are negative, thus \eqref{interm14bis} remains in force. Moreover, with the experience gained in part (b), we fix
\begin{equation}\label{interm15bis}
T=(\kappa^0)^{-1/\delta}>1.
\end{equation}
With the choices in \eqref{interm14bis} and \eqref{interm15bis}, we observe that, on the one hand,
$$
(t+T)^{-\delta(p-1)+2\gamma}\leq T^{-\delta(p-1)+2\gamma}<1, \quad t\geq0,
$$
and on the other hand
$$
(t+T)^{-2\gamma+\delta(q-1)}\leq T^{-2\gamma+\delta(q-1)}<1, \quad t\geq0.
$$
This ensures that $T_2>0$ and $T_3>0$ for any $t\geq0$. Moreover, for $t=0$, we have the comparison between initial conditions:
\begin{equation}\label{interm19}
\Phi(x,0)=T^{-\delta}\psi_0(T^{-\gamma}x)=\kappa^0\psi_0(T^{-\gamma}x)\geq u_0(T^{-\gamma}x).
\end{equation}
It only remains to choose $\delta>0$ such that $T_1>0$ too. Noticing that \eqref{interm14bis} and \eqref{limits} imply
$$
\frac{\delta}{\gamma}<\frac{2}{q-1}=-\lim\limits_{\zeta\to\infty}\frac{\zeta\psi_0'(\zeta)}{\psi_0(\zeta)},
$$
we infer that there is $R_0>0$ such that $T_1>0$ for $\zeta\in\real$ such that $|\zeta|>R_0$. For the compact interval $[-R_0,R_0]$, the proof follows by compensating $T_1$ by $T_2$ in order to choose $\delta$, following exactly the same lines as in the final part of the proof of part (b), this compensation, together with the choices in \eqref{interm14bis} and \eqref{interm15bis} leading to the condition
$$
\delta<\min\left\{\frac{1}{2}\left(\frac{1}{p-1}+\frac{1}{q-1}\right),\left[1-(\kappa^0)^{(p-1)(p-q)/(p+q-2)}\right]L_0^{p-1}\right\},
$$
with $L_0$ defined as in \eqref{interm20}, noticing that the previous election also entails that $\gamma\in(0,1/2)$. The comparison principle and \eqref{interm19} then entail that
$$
\Phi(x,t)=(T+t)^{-\delta}\psi_0(\zeta)\geq u(\zeta,t), \quad (\zeta,t)\in\real\times(0,\infty),
$$
and in particular
$$
\|u(t)\|_{\infty}\leq (T+t)^{-\delta}\|\psi_0\|_{\infty}<1,
$$
provided $t\geq t_0$ sufficiently large. This fact together with part (a) lead to the conclusion.
\end{proof}

\section*{Discussion and extensions}

The previous analysis raises a few questions related to a finer description of the dynamics of solutions to Eq. \eqref{eq1}, that we comment in the next lines.

\medskip

$\bullet$ \textbf{Blow-up rates and patterns.} Due to the fact that, for large values of $u$ (close to a blow-up point), $u^q$ is negligible compared to $u^p$, it is strongly expected that the blow-up behavior of solutions to Eq. \eqref{eq1} presenting finite time blow-up is totally similar to the one of the usual reaction-diffusion equation $u_t=u_{xx}+u^p$. In fact, if $T\in(0,\infty)$ is the blow-up time of the solution to the Cauchy problem \eqref{eq1}-\eqref{ic} with an initial condition $u_0$ satisfying either Theorem \ref{th.separ.equal} (a) or Theorem \ref{th.separ.upper} (b), then it cannot lie completely below the solution given by either \eqref{sol.time} (if $q=1$) or to \eqref{interm9} (if $q\in(1,p)$) having the same blow-up time $T$. We then deduce the lower blow-up rate from \eqref{blowup.rate.upper}, that is,
$$
\|u_0\|_{\infty}\geq C(T-t)^{-1/(p-1)}, \quad t\in(0,T),
$$
for some $C>0$. As for the upper rate and other fine properties of blow-up, the theory in \cite{QS} should apply without many changes. In particular, for $q=1$ the previous rate follows as a particular case of the results in the recent paper \cite{Zhang24}.

$\bullet$ \textbf{Large time behavior for solutions with decay as $t\to\infty$.} As we have seen, for $q=1$ we have obtained the large time behavior by an asymptotic simplification. The same must hold true for $q\in(1,p)$ since for $u$ very small, the reaction $u^p$ is negligible with respect to the absorption $u^q$, so that the asymptotic simplification is practically obvious at a formal level. We do not enter this discussion in detail, since the large time behavior of solutions to the pure absorption-diffusion equation (neglecting $u^p$) is not straightforward, but given by either a Gaussian, as a further asymptotic simplification, if $q>3$, or some very singular solutions constructed in \cite{BPT86} if $1<q<3$.

$\bullet$ \textbf{Extinction behavior for $q\in(0,1)$}. As established in \cite{HBS14}, for $q\in(0,1)$, initial conditions $u_0$ lying below the stationary solution lead to finite time extinction, that is, there exists $T_e\in(0,\infty)$ such that $u(t)\not\equiv0$ for $t\in(0,T_e)$, but $u(T_e)\equiv0$. Establishing the extinction rate (that is not studied in \cite{HBS14}) follows rather straightforwardly, since one can construct solutions depending only on time $h(t)$, solving the differential equation
$$
h'(t)=h^p(t)-h^q(t), \quad h(0)<1,
$$
and, by estimates in the same style as the ones leading to \eqref{blowup.rate.upper}, find that
$$
[(1-q)(1-h(0)^{p-q})]^{1/(1-q)}(T_e-t)^{1/(1-q)}<h(t)<(1-q)^{1/(1-q)}(T_e-t)^{1/(1-q)}.
$$
Moreover, a solution to Eq. \eqref{eq1} is a supersolution to the pure absorption-diffusion equation obtained by neglecting $u^p$, and about which it is known that the extinction rate of general solutions is
$$
\|u(t)\|_{\infty}\sim C(T_e-t)^{1/(1-q)}, \quad {\rm as} \ t\to T_e,
$$
so that the same rate will be in force for Eq. \eqref{eq1} with $p>1$ and $q\in(0,1)$. However, a finer analysis of the extinction phenomenon is rather complex, as it should be at least similar to (if not technically more involved than) the one for the absorption-diffusion equation
$$
u_t=u_{xx}-u^q, \quad 0<q<1,
$$
studied in detail in classical but quite complicated works such as \cite{GHV, HV92}. We leave this discussion here.

\bigskip

\noindent \textbf{Acknowledgements} R. G. I. and A. S. are partially supported by the Project PID2020-115273GB-I00 and by the Grant RED2022-134301-T funded by MCIN/AEI/10.13039/ \\ 501100011033 (Spain).

\bibliographystyle{plain}

\end{document}